\documentclass[11pt]{article}

\usepackage[a4paper, left=2.5cm,top=2.2cm,right=2.5cm,bottom=2.7cm]{geometry}
\title{A New Sparsification and Reconstruction Strategy for Compressed Sensing Photoacoustic Tomography}

\usepackage[blocks]{authblk}

\author{Markus Haltmeier\thanks{Correspondence: {\tt markus.haltmeier@uibk.ac.at}}\,{}} 
\author{Michael Sandbichler}
\affil{Department of Mathematics, University of Innsbruck\\
Technikestra{\ss}e 13, 6020 Innsbruck, Austria}

\author{Thomas Berer}
\author{Johannes Bauer-Marschallinger}
\author{Peter Burgholzer\thanks{Also affiliated with Christian Doppler Laboratory for Photoacoustic Imaging and Laser Ultrasonics, Linz}}
\affil{Research Center for Non-Destructive Testing (RECENDT)\\
Altenberger Stra{\ss}e 69, 4040 Linz, Austria}

\author{Linh V. Nguyen}
\affil{Department of Mathematics, University of Idaho\\
875 Perimeter Dr, Moscow, ID 83844, US}

\affil{}
\date{}

\usepackage{psfrag}
\usepackage{graphicx}
\usepackage{hyperref}
\usepackage{siunitx}
\usepackage{amsthm,amsmath,amssymb}

\usepackage{soul}
\usepackage{xcolor}
\colorlet{lred}{red!40}
\colorlet{lgreen}{green!40}
\colorlet{lblue}{blue!40}

\newcommand{\rr}{\mathbf r}
\newcommand{\data}{y}
\newcommand{\source}{f}
\newcommand{\Lsource}{h}
\newcommand{\R}{\mathbb R}

\newcommand{\kl}[1]{(#1)}
\newcommand\norm[1]{\Vert#1\Vert}
\newcommand{\set}[1]{\{#1\}}
\newcommand{\edot}{\cdot}
\newcommand{\prox}{\mathrm{prox}}
\newcommand{\sign}{\mathrm{sign}}

\newcommand{\fullop}{\mathbf M}
\newcommand{\wave}{\mathbf W}
\newcommand{\Am}{\mathbf A}

\usepackage{amsthm}
\newtheorem{theorem}{Theorem}[section]
\newtheorem{remark}{Remark}[section]

\begin{document}
\maketitle

\begin{abstract}
	Compressed sensing (CS) is a promising approach to reduce the number of measurements in photoacoustic  tomography (PAT) while preserving  high spatial resolution. This allows to
	increase  the measurement  speed and to reduce system costs. Instead of collecting
	point-wise measurements, in CS one uses various combinations of
	pressure values  at different sensor locations. Sparsity is the main condition allowing
	to recover the photoacoustic  (PA) source from compressive measurements. In this paper
	we introduce a new concept  enabling sparse recovery in CS PAT.
	Our approach is based on the fact that the second time derivative applied to the
	measured pressure data    corresponds to the application of the Laplacian to
	the original PA source. As typical PA sources consist of smooth parts and singularities along
	interfaces the Laplacian of the source is sparse (or at least compressible). To efficiently exploit the induced sparsity we develop a  reconstruction framework to jointly recover the initial and the modified sparse source. Reconstruction results with simulated as well as experimental data are given.
\end{abstract}

\section{Introduction}\label{sec:intro}

Photoacoustic tomography (PAT) is a non-invasive hybrid imaging technology, that beneficially
combines  the high contrast of pure optical imaging and the high spatial resolution of pure ultrasound imaging (see \cite{Bea11,Wan09b,XuWan06}).
The basic principle of PAT is  as follows (see Fig.~\ref{fig:pat}). A semitransparent sample (such as a part of a human patient) is illuminated with short pulses of optical radiation. A fraction  of the optical energy is absorbed inside the sample which causes thermal heating, expansion, and a subsequent acoustic pressure wave  depending on the interior absorbing structure of the sample.
The acoustic pressure is measured outside of the  sample and used  to reconstruct an image of the interior.

\begin{psfrags}
	\begin{figure}[htb!]
		\begin{center}
			\includegraphics[width=1\columnwidth]{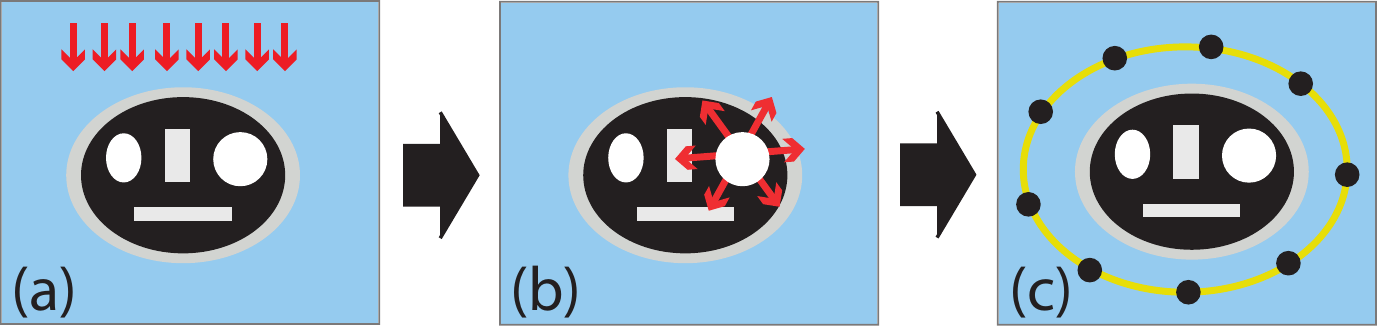}
			\caption{\label{fig:pat} (a) An object is illuminated with a  short optical pulse;  (b) the absorbed light distribution causes an acoustic pressure; (c) the  acoustic pressure is measured outside the object and used to reconstruct an image of the interior.}
		\end{center}
	\end{figure}
\end{psfrags}

In this paper we consider  PAT in heterogeneous acoustic media, where the
acoustic  pressure  satisfies the wave equation
\begin{equation} \label{eq:wave}
\frac{\partial^2 p(\rr, t)}{\partial t^2}  -  c^2(\rr)\Delta_\rr p(\rr, t)  =
\delta'(t) \, \source (\rr) \,.
\end{equation}
Here  $\rr  \in \R^d$ is the spatial location, $t \in \R$ the time variable, $\Delta_\rr$ the spatial Laplacian, $c(\rr)$ the speed of sound, and $\source (\rr)$ the  photoacoustic (PA) source that has to be recovered.
The wave equation \eqref{eq:wave} is augmented with  initial condition
$p(\rr, t) =0$  on $\set{t < 0}$. The acoustic pressure is then uniquely defined and referred to as the causal solution of~\eqref{eq:wave}.
Both cases $d  =2,3 $  for the   spatial dimension are relevant in PAT:
The case $d=3$ arises in PAT using classical point-wise measurements; the case $d=2$
is relevant   for PAT with integrating  line detectors~\cite{Bauer-Marschallinger:17,BurBauGruHalPal07,PalNusHalBur07a,paltauf2017piezoelectric}.

To recover the PA  source, the pressure is measured  with sensors
distributed on a surface or curve outside of the sample; see Fig.~\ref{fig:pat}.
Using standard  sensing, the spatial sampling step size limits the spatial resolution
of the measured  data  and therefore the spatial resolution of the final reconstruction.
Consequently, high spatial resolution requires a large number of detector  locations.
Ideally, for high frame rate, the pressure data are measured in parallel with a large
array made of small  detector elements.  However, producing a detector array with  a large number
of parallel  readouts is costly and technically demanding. In this work we use techniques
of compressed sensing (CS) to  reduce the number of required measurements and thereby
accelerating  PAT while keeping high spatial resolution \cite{sandbichler2015novel,arridge2016accelerated,betcke2016acoustic,haltmeier2016compressed}.

CS is a new sensing paradigm introduced in \cite{CanRomTao06a,CanTao06,Don06} that allows to capture high resolution signals using   much less measurements than advised by Shannon's sampling theory. The basic idea is to replace point samples by linear measurements, where each measurement consists of a  linear combination of sensor values. It  offers the ability  to reduce  the number of measurements while keeping  high spatial resolution.
One crucial ingredient enabling CS PAT is sparsity,
which refers to the requirement that the  unknown signal is sparse,
in the sense that it has only a small number of entries  that are significantly
different from zero (possibly after a change of basis).

\subsection{Main contributions}
\label{ssec:outline}

In this work we develop a new  framework for CS PAT that allows to bring sparsity into play. Our approach is rooted in the concept of sparsifying  temporal transforms developed for PAT in \cite{sandbichler2015novel,haltmeier2016compressed}
for two and three  spatial dimensions. However, the approach
in this present paper extends and simplifies this transform approach considerably.
First, it equally applies to any  detection surface and  arbitrary spatial dimension.
Second,
the new method can even be applied  to heterogenous media.
In order to achieve this, we use the second time derivative applied to the pressure data as a sparsifying transform. Opposed to \cite{sandbichler2015novel,haltmeier2016compressed},  where the transform was used to sparsify the measured signals, in the present work we exploit this for
obtaining sparsity in the original imaging domain.

Our new approach  is based on the  following. Consider the  second time derivative
$\partial_t^2 p(\rr, t)$ of the  PA pressure. We will show, that this transformed pressure
again satisfies the wave equation,  however
with the modified  PA source
$c^2 \Delta_{\rr} \source$ in place of  the original
PA source $\source$.
If the original
PA source  consists of smooth parts  and jumps,
the modified source consists of smooth parts and sparse  structures;
see Fig.~\ref{fig:pex} for an example.
This enables the use of efficient CS reconstruction algorithms
based on sparse recovery.
One  possible  approach is based on the following two-step procedure. First, recover  an approximation
$\Lsource (\rr) \simeq \Delta_{\rr} \source (\rr) $ via $\ell^1$-minimization. In a second step, recover
an approximation  to $\source$ by solving the Poisson equation $\Delta_{\rr} \source  =    \Lsource/c^2 $.

\begin{psfrags}
	\begin{figure}[htb!]
		\begin{center}
			\includegraphics[width=\columnwidth]{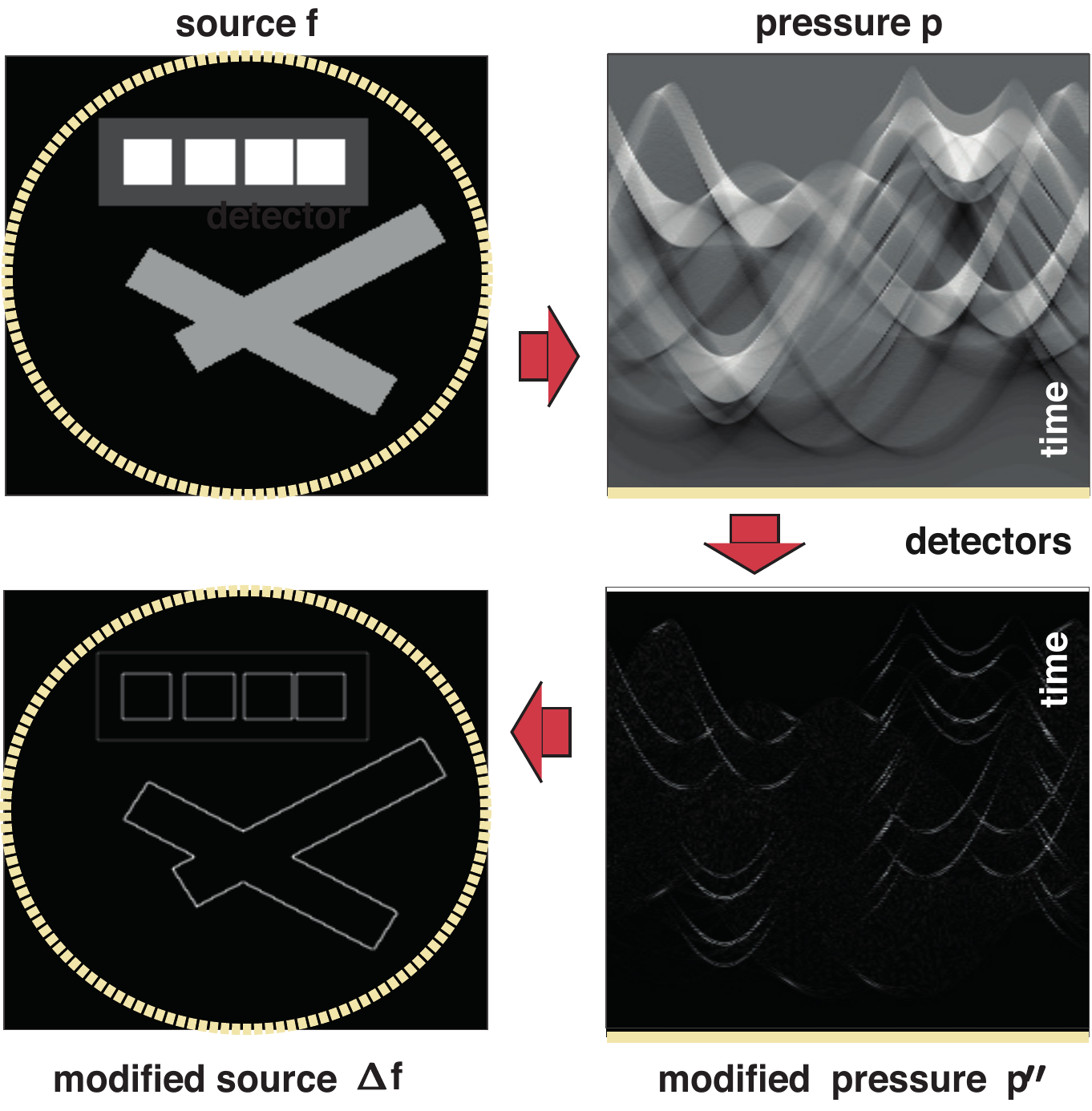}
			\caption{\label{fig:pex} Top left: PA source $\source$ (circle indicates the detector locations). Top right:  Modified source $ \Delta_{\rr} \source$. Bottom:   pressure   $p$ and second derivative $\partial_t^2 p$ (detector locations in horizontal direction and time in vertical direction).}
		\end{center}
	\end{figure}
\end{psfrags}

While the above two-stage  approach turned out to very well recover the
singularities of  the PA source,
at the same time it shows disturbing low-frequency artifacts. Therefore, in this paper
we develop a completely novel  strategy for jointly recovering $\source$ as well as its
Laplacian. It is based on solving the constrained optimization problem
\begin{equation*}
	\begin{aligned}
		&\min_{(\source, \Lsource)}    \norm{\Lsource}_1  +  I_{C} (\source)  \\
		&\text{such that }
		\begin{bmatrix} \fullop \source, \fullop\Lsource , c^2 \Delta_{\rr} \source  -    \Lsource \end{bmatrix} =
		\begin{bmatrix} \data , \data'' ,0 \end{bmatrix}   \,.
	\end{aligned}
\end{equation*}
Here $\fullop$ is the forward operator (including wave propagation and compressed
measurements),
$\norm{\edot }_1$ denotes the $\ell^1$-norm guaranteeing
sparsity of $\Lsource$ and $I_C $ the indicator
function of the positive cone $ C  \triangleq \set{\source \mid  \source(\rr) \geq 0 }$
guaranteeing  non-negativity.
In the  case of noisy data we consider a penalized  version that can be
efficiently solved via various modern optimization techniques, such as
forward backward splitting.

\subsection{Outline}
\label{ssec:outline}
In Section \ref{sec:CSPAT} we describe PAT and  existing
CS approaches. We thereby focus on the role of sparsity
in PAT. The proposed framework for CS PAT will be presented in
Section~\ref{sec:proposed}. This includes the sparsification of the PA source and
the joint reconstruction algorithm. Numerical and experimental results are presented
in Section~\ref{sec:results}. The paper concludes with a summary and outlook on future research presented in  Section \ref{sec:conclusion}.

\section{Compressed photoacoustic tomography}

\label{sec:CSPAT}

\subsection{Photoacoustic tomography}

Suppose that $V \subseteq \R^d$ is a bounded region
and let $\rr_i$ for $i=1, \dots, n$ denote admissible
detector locations distributed on the  boundary $\partial V$.
Then, for a  given source $f$ supported inside $V$,
we define
\begin{equation} \label{eq:data}
(\wave \source) (\rr_i,  t) \triangleq p \kl{\rr_i, t}
\qquad \text{ for } i = 1, \dots, n \, \text{ and }t \in [0,T]  \,,
\end{equation}
where  $p(\rr,t)$ is the solution of \eqref{eq:wave}. The inverse source problem of PAT is to recover the PA source $\source$ from the $\wave \source$ in \eqref{eq:data}. The measured data $\wave \source$ is considered to be fully/completely sampled if the transducers are densely located on the whole boundary $\partial V$, such that the function $f$ can be stably reconstructed from the data. Finding necessary and sufficient sampling conditions for PAT is still on-going research \cite{haltmeier2016sampling}.

Let us mention that most of the theoretical works on PAT consider the continuous setting where the transducer locations are all points of a surface or curve $\Gamma \subseteq \partial V$;
see~\cite{FPR,XW05,Kun07,FHR,Kun07,IPI,Halt2d,Halt-Inv,natterer2012photo,Pal-Uniform,haltmeier2017analysis}. On the other hand, most works on discrete settings consider both discrete spatial and time variables \cite{huang2013full,haltmeier2016sampling}.  The above setting \eqref{eq:data} has been considered in a few works \cite{sandbichler2015novel,haltmeier2016compressed,chung2017motion}. It well  reflects the high sampling rate of the time variable in many practical PAT systems.


\subsection{Compressive  measurements in PAT}

The number $n$ of detector positions in  \eqref{eq:data} is directly related to the resolution of the final reconstruction. Namely, consider the case $V$ being the
disc of radius $R$ and $\source$ being essentially  wavenumber limited with maximal wavenumber
give by $\lambda_0$. Then, $N_\varphi \geq 2 R \lambda_0$ equally spaced transducers are sufficient to recover $\source$ with small error; see \cite{haltmeier2016sampling}. In  many biomedical applications, this results in a high  sampling rate. For example, the PA source may contain narrow features such as blood vessels and have sharp interfaces. This  results in large wavenumbers and necessary high sampling rate. For this reason, full sampling  in PAT is costly and time consuming.

To reduce the number of measurements  while  preserving resolution,  we use CS measurements in PAT. Instead of collecting $n$
individually sampled signals as in \eqref{eq:data}, we take CS
measurements
\begin{equation} \label{eq:cs}
\data(j, t)
\triangleq
(\fullop \source)(j, t)
\triangleq
(\Am \wave)(j, t)
=	\sum_{i=1}^n 	
\Am[j, i] p(\rr_i, t)
\quad \text{ for } j \in \set{  1, \dots, m} \,,
\end{equation}
with  $m \ll n$.
In \cite{burgholzer2016sparsifying,sandbichler2015novel}  we proposed to take the
measurement matrix $\Am$  in \eqref{eq:cs}  as the adjacency matrix of a lossless expander graph. Hadamard matrices have been proposed in
\cite{betcke2016acoustic,huynh2016single}. In this work, we  take $\Am$ as a random Bernoulli matrix with entries $\pm 1/\sqrt{m}$ with equal probability or a Gaussian random matrix consisting of i.i.d. $\mathcal{N}(0,{1}/{{m}})$-Gaussian random variables in each entry. These choices are validated by the fact that Gaussian and Bernoulli random matrices satisfy the restricted isometry property (RIP) with high probability (see Section~\ref{S:sparsity} below).

\subsection{The role of sparsity} \label{S:sparsity}

A central aspect in the theory of CS  is sparsity of the given data in some basis or frame~\cite{CanRomTao06a,Don06,foucart2013mathematical}. Recall that a vector $x\in\R^N$ is called
$s$-sparse if $\lvert \{i \mid x_i \neq0 \} \rvert \leq s$ for some number $s\ll N$, where $\lvert \,\dot \, \rvert$ is used to denote the number of elements of some
set.
If the data is known to have sparse structure, then reconstruction procedures using $\ell_1$-minimization or greedy-type methods can often be guaranteed to yield high quality results even if the problem is severely ill-posed~\cite{CanRomTao06a,GraHalSch11}.
If we are given measurements $\fullop x = y$, where $x\in\R^N$ and $y\in\R^m$ with $m\ll N$, the success of the aforementioned reconstruction procedures can for example be guaranteed if the matrix $A$ satisfies the restricted isometry property (RIP), i.e. for all $s$-sparse vectors $z$ we have
\begin{equation} \label{eq:RIP}
(1-\delta)\|z\|^2\leq \| \fullop z\|^2\leq(1+\delta)\|z\|^2 \,,
\end{equation}
for an RIP constant $\delta < 1/\sqrt{2}$; see~\cite{foucart2013mathematical}.
Gaussian and Bernoulli random matrices satisfy the RIP with high probability, provided $m \geq C s \log(\mathrm{e} n/s)$ for some reasonable constant $C$ and with $\mathrm{e}$ denoting Euler's number~\cite{baraniuk2008simple}.

In PAT, the possibility to sparsify the data has recently been examined~\cite{sandbichler2015novel,haltmeier2016compressed}.
In these works it was observed that the measured pressure data could be sparsified and the sparse reconstruction methods were applied directly to the pressure data. As a second step, still a classical reconstruction via filtered backprojection had to be performed.
The sparsification of the data was achieved with a transformation in the time direction of the pressure data. In two dimensions, the transform is a first order pseudo-differential operator \cite{sandbichler2015novel}, while in three dimensions the transform is of second order \cite{haltmeier2016compressed}.

\section{Proposed Framework}
\label{sec:proposed}


\subsection{Sparsifying transform}

The following theorem is the foundation
of our CS  approach.
It shows that the required sparsity
in space can be  obtained by applying the
second time derivative to the measured data.

\begin{theorem} \label{thm:laplace}
	For given source term $\source \colon \R^d \to \R$,  let  $p(\rr,t)$ denote the causal
	solution of the wave equation~\eqref{eq:wave}.  Then  $\partial_t^2 p$  is the
	causal solution of
	\begin{equation} \label{eq:wave2}
	\frac{\partial^2 p''}{\partial t^2}  -  c^2 (\rr)  \Delta_\rr p'' (\rr, t)  =
	\delta'(t) \,  c^2(\rr) \Delta_{\rr} \source (\rr) \,.
	\end{equation}
	In particular, $\partial_t^2 \fullop  [\source] = \fullop  [c^2 \Delta_{\rr} \source]$,
	where $\fullop$ denotes the CS PAT forward  operator defined
	by \eqref{eq:cs}.
\end{theorem}

\begin{proof}
	We first  recall that the solution of \eqref{eq:wave} for $t>0$  is
	equivalent to the initial value problem
	\begin{align}  \label{eq:wave-ini1}
		(\partial^2_t     -  c^2 \Delta) p(\rr,t)
		&=
		0 \,, \text{ for }
		\kl{\rr,t} \in
		\R^d \times \kl{0, \infty}
		\\ \label{eq:wave-ini2}
		p\kl{\rr,0}
		&=
		\source(\rr) \,,
		\quad \text{ for }
		\rr  \in \R^d
		\\ \label{eq:wave-ini3}
		\partial_t
		p\kl{\rr,0}
		&=0 \,,
		\quad \text{ for }
		\rr  \in \R^d \,.
	\end{align}
	To see this equivalence note first that the solution of
	\eqref{eq:wave-ini1}-\eqref{eq:wave-ini3}   extends to a smooth
	function $p\colon \R^d \times \R \to \R$ that is even in the time
	variable. Denoting by $\chi = \chi\set{t>0}$ the characteristic function
	function of $(0,\infty)$ yields
	$(\partial^2_t     -  c^2  \Delta_{\rr} ) (\chi p) = \chi (\partial^2_t -  c^2  \Delta_{\rr} ) ( p)
	+ \delta' \source  + 2 \delta (\partial_t  p) $
	Using the initial conditions \eqref{eq:wave-ini2}, \eqref{eq:wave-ini3}
	shows that $\chi p$ (which coincides with $p$ for positive times)
	solves \eqref{eq:wave}.

	According to the above considerations, for $t>0$, the function
	$p$ coincides with the solution of   \eqref{eq:wave-ini1}-\eqref{eq:wave-ini3}.
	As a consequence, for $t>0$, the function $q = \partial_t^2 p$ satisfies
	\begin{align}  \label{eq:wave2-ini1}
		(\partial^2_t     -  c^2 \Delta) q(\rr,t)
		&=
		0 \,, \text{ for }
		\kl{\rr,t} \in
		\R^d \times \kl{0, \infty}
		\\ \label{eq:wave2-ini2}
		q\kl{\rr,0}
		&=
		c^2(\rr) \Delta_{\rr} \source(\rr) \,,
		\quad \text{ for }
		\rr  \in \R^d
		\\ \label{eq:wave2-ini3}
		\partial_t
		q\kl{\rr,0}
		&=0 \,,
		\quad \text{ for }
		\rr  \in \R^d \,.
	\end{align}
	In fact, \eqref{eq:wave2-ini1} results by applying the second derivative  to
	\eqref{eq:wave-ini1} and \eqref{eq:wave2-ini3} follows from the symmetry of
	$p$. Evaluating \eqref{eq:wave-ini1} at $t=0$ and applying the Laplacian to
	\eqref{eq:wave-ini2} yields initial condition \eqref{eq:wave2-ini2}.  Finally as for the
	original pressure one concludes that  \eqref{eq:wave2-ini1}-\eqref{eq:wave2-ini3}
	implies \eqref{eq:wave2}.
\end{proof}


\begin{remark} \label{rem:discrete}
	In order to focus on the main ideas, throughout
	the following we assume the spatial variable
	$\rr  \in \set{0, \dots, N_\rr}^d$ to be already discretized.
	The discrete Laplacian $\Delta_\rr$ then may be defined  via
	symmetric  finite differences; alternatively $\Delta_\rr$ may be defined
	via the Fourier transform in the spectral domain.
	Theorem \ref{thm:laplace} and the equivalence of \eqref{eq:wave} and
	\eqref{eq:wave-ini1}-\eqref{eq:wave-ini3}  then holds for discrete
	sources   $f \colon \set{0, \dots, N_\rr}^d \to \R$ as well.
\end{remark}

Typical  phantoms  consist of smoothly varying parts
and rapid changes  at interfaces. For such PA sources,
the modified source  $c^2 \Delta_{\rr} \source$
is sparse or at least compressible. The theory of
CS therefore predicts that the modified source
can be recovered by solving
\begin{equation} \label{eq:L1exact}
\min_{ \Lsource}     \norm{\Lsource}_1
\quad \text{such that }     \fullop \Lsource  =   \partial_t^2 \data       \,.
\end{equation}
In the case the unknown  is only  approximately sparse or the  data
are  noisy, one instead minimizes the penalized functional
problem $\frac{1}{2}  \norm{\fullop \Lsource  -   \partial_t^2 \data}_2^2
+\beta \norm{\Lsource}_1$, where $\beta >0$ is a regularization parameter
which gives trade-off between the data-fitting term $\norm{\fullop \Lsource  -   \partial_t^2 \data}_2^2$ and the regularization term  $\norm{\Lsource}_1$. Having obtained an approximation of $\Lsource$
by either solving \eqref{eq:L1exact} or  the relaxed version,  one can recover
the original PA source $\source $ by subsequently solving the Poisson
equation $   \Delta_{\rr} \source  = \Lsource/c^2$ with zero
boundary conditions.

While the  above two-stage procedure recovers boundaries well,
we  observed disturbing low frequency artifacts
in the reconstruction of $f$. Therefore, below we
introduce   a new  reconstruction approach based on Theorem~\ref{thm:laplace}
that jointly recovers $\source$ and $\Lsource$.

\subsection{Joint reconstruction approach}

As argued above, the second derivative $p''$
is  well suited (via $ c^2(\rr) \Delta_{\rr} \source$ ) to recover
the singularities of  $\source$, but hardly contained low-frequency
components of $\source$. On the other hand, the low frequency
information is contained in the original data, which is still available to us.
Therefore we propose the following joint constrained optimization problem
\begin{equation} \label{eq:joint2}
\begin{aligned}
&\min_{(\source, \Lsource)}  \norm{\Lsource}_1  +  I_{C} (\source)  \\
&\text{such that }
\begin{bmatrix}  \fullop\source,  \fullop\Lsource ,  \Delta_{\rr} \source  -    \Lsource/c^{2} \end{bmatrix} =
\begin{bmatrix} \data , \data'' ,0 \end{bmatrix}   \,.
\end{aligned}
\end{equation}
We habe the following result.

\begin{theorem}
	Assume that $\source \colon  \set{0, \dots, N_\rr}^d \to \R$  is
	non-negative and that  $\Delta_{\rr} \source$ is $s$-sparse.
	Moreover, suppose  that  the measurement matrix $\fullop$  satisfies the RIP
	(see \eqref{eq:RIP})  and denote $\data = \fullop\source$.
	Then, the pair $[\source, \Delta_{\rr} \source]$ can be recovered as the
	unique solution of~\eqref{eq:joint2}.
\end{theorem}

\begin{proof}
	According to Theorem~\ref{thm:laplace} (in a discrete form; compare Remark~\ref{rem:discrete}),
	$p''(\rr, t)$  is the unique causal
	solution of the wave equation \eqref{eq:wave2} with modified source term
	$\Lsource = c^2(\rr) \Delta_{\rr} \source$. As a consequence $c^2(\rr) \Delta_{\rr} \Lsource$ satisfies
	$\fullop\Lsource = \data'' $, which implies that the pair  $[\source, \Lsource]$
	is a feasible for~\eqref{eq:joint2}. It remains to verify that $[\source, \Lsource]$ is the only
	solution of~\eqref{eq:joint2}.  To show  this, note  that for  any solution
	$[\source^*, \Lsource^*]$ of~\eqref{eq:joint2} its second component $\Lsource^*$
	is a solution of \eqref{eq:L1exact}.
	Because $\fullop$ satisfies the $s$-RIP, and
	$\Lsource = c^2(\rr) \Delta_{\rr} \source$ is  $s$-sparse, CS
	theory implies that \eqref{eq:L1exact} is uniquely solvable \cite{CanRomTao06a,Don06,foucart2013mathematical} and therefore
	$\Lsource = \Lsource^*$. The last constraint then  implies that
	$ \source^* = \source$.
\end{proof}

In the case the data only approximately sparse or noisy, we propose, instead of
\eqref{eq:joint2}, to solve the    $\ell^2$-relaxed version
\begin{equation} \label{eq:joint-pen}
\frac{1}{2} \norm{\fullop \source- \data}_2^2
+
\frac{1}{2} \norm{\fullop\Lsource - \data''}_2^2
+
\frac{\alpha}{2} \norm{\Delta_{\rr} \source  -    \Lsource/c^{2}}_2^2
+
\beta \norm{\Lsource}_1  +  I_{C} (\source)
\to \min_{(\source, \Lsource)} \,.
\end{equation}
Here $\alpha>0$ is a tuning and  $\beta >0$ a regularization
parameter. There are several modern methods to efficiently solve
\eqref{eq:joint-pen}. In this work we use the forward-backward splitting
with quadratic term as smooth part used in  the explicit (forward) step and
$ \beta \norm{\Lsource}_1  +  I_{C} (\source) $ as non-smooth part used for
the implicit (backward) step.

\subsection{Numerical minimization}

We will solve~\eqref{eq:joint-pen} using a proximal gradient
algorithm~\cite{combettes2011proximal}, which is an  algorithm
well suited for minimizing  the sum  of a smooth and a non-smooth but convex part.
In the case of  \eqref{eq:joint-pen}  we take the smooth part as
\begin{equation}
\Phi(\source,\Lsource) \triangleq  \frac{1}{2} \norm{\fullop \source- \data}_2^2
+
\frac{1}{2} \norm{\fullop\Lsource - \data''}_2^2
+
\frac{\alpha}{2} \norm{\Delta_{\rr} \source  -    \Lsource/c^{2}}_2^2
\end{equation}
and the  non-smooth part as $\Psi(f,h) \triangleq \beta \norm{\Lsource}_1  +  I_{C}(f)$.

The proximal gradient  algorithm then alternately  performs an explicit gradient step
for $\Phi$ and an implicit proximal step for $\Psi$.
The gradient $[\nabla_\source \Phi, \nabla_\Lsource \Phi]$ of the  smooth part can
easily be computed  to be
\begin{align*}
	\nabla_f \Phi (f,h)
	&=  \fullop^* (\fullop \source- \data)-  \alpha \Delta_{\rr} (\Delta_{\rr} \source  -    \Lsource/c^{2})\\
	\nabla_h \Phi (f,h)
	&=   \fullop^* (\fullop \Lsource- \data'')- \frac{\alpha }{c^2}(\Delta_{\rr} \source  -    \Lsource/c^{2}) \,.
\end{align*}
The proximal map of the non-smooth part is given by
\begin{align*}
	\prox(\source,\Lsource) &= [\prox_{I_C}(\source), \prox_{\beta\|\cdot\|_1(\Lsource)}] \,,\\
	\prox_{I_C}(\source) &= (\max(\source_i,0))_i \,,\\
	\prox_{\beta\|\cdot\|_1}(h)  &= (\max(|\Lsource_i|-\beta,0)\,\sign(\Lsource_i))_i
\end{align*}
With this, the proximal gradient algorithm is given by
\begin{align} \label{eq:prox1}
	\source^{k+1} &= \prox_{I_C}\left(\source^k - \mu_k  \nabla_f \Phi (\source^k, \Lsource^k)  \right)
	\\ \label{eq:prox2}
	\Lsource^{k+1} &= \prox_{\beta\|\cdot\|_1}\left(\Lsource^k - \mu_k  \nabla_k \Phi (\source^k, \Lsource^k)\right),
\end{align}
where $(\source^k, \Lsource^k)$ is the $k$-th iterate and $\mu_k$ the step size in the $k$-th
iteration. We initialize the proximal gradient algorithm with $\source^0=\Lsource^0=0$.

\begin{remark}
	Note that the optimization problem~\eqref{eq:joint2} is further
	equivalent to the analysis-$\ell_1$ problem
	\begin{equation} \label{eq:jointAnalysis}
	\begin{aligned}
	&\min_{\source}  \norm{c\Delta_{\rr} \source}_1  +  I_{C} (\source)  \\
	&\textnormal{such that }
	\fullop\source=
	\data\,.
	\end{aligned}
	\end{equation}
	Implementation of \eqref{eq:jointAnalysis} avoids taking the second time derivative
	of the data $y$.  Because the proximal map of $\source \mapsto \norm{c\Delta_{\rr} \source}_1$
	in not available explicitly, \eqref{eq:jointAnalysis}  and its  relaxed versions
	cannot be  straightforwardly  addressed with the  proximal gradient algorithm.
	Therefore, in the present paper we only use the model \eqref{eq:joint-pen}
	and the  algorithm \eqref{eq:prox1}, \eqref{eq:prox2} for its minimization.
	Different  models and algorithms will be  investigated  in future research.
\end{remark}

\section{Experimental and numerical results}
\label{sec:results}

\subsection{Numerical results}

For the presented numerical results, the two dimensional PA source term
$\source \colon \set{0, \dots, N_{\rr}}^2 \to \R$ depicted in Figure~\ref{fig:pex}
is used which is assumed to be  supported in a disc of radius $R$.
Additional results are presented using an MRI image.
The synthetic data is recorded on the  boundary circle of radius $R$, where the the time was discretized
with  $301$ equidistant sampling points in the interval $[0,2R]$. The full data was recorded
at $n=200$ detector locations. The reconstruction of both phantoms via the filtered backprojection
algorithm of \cite{FPR} from the  full measurements is shown in Figure~\ref{fig:fbpRecon}.

\begin{psfrags}
	\begin{figure}[htb!]
		\begin{center}
			\includegraphics[width=0.8\columnwidth]{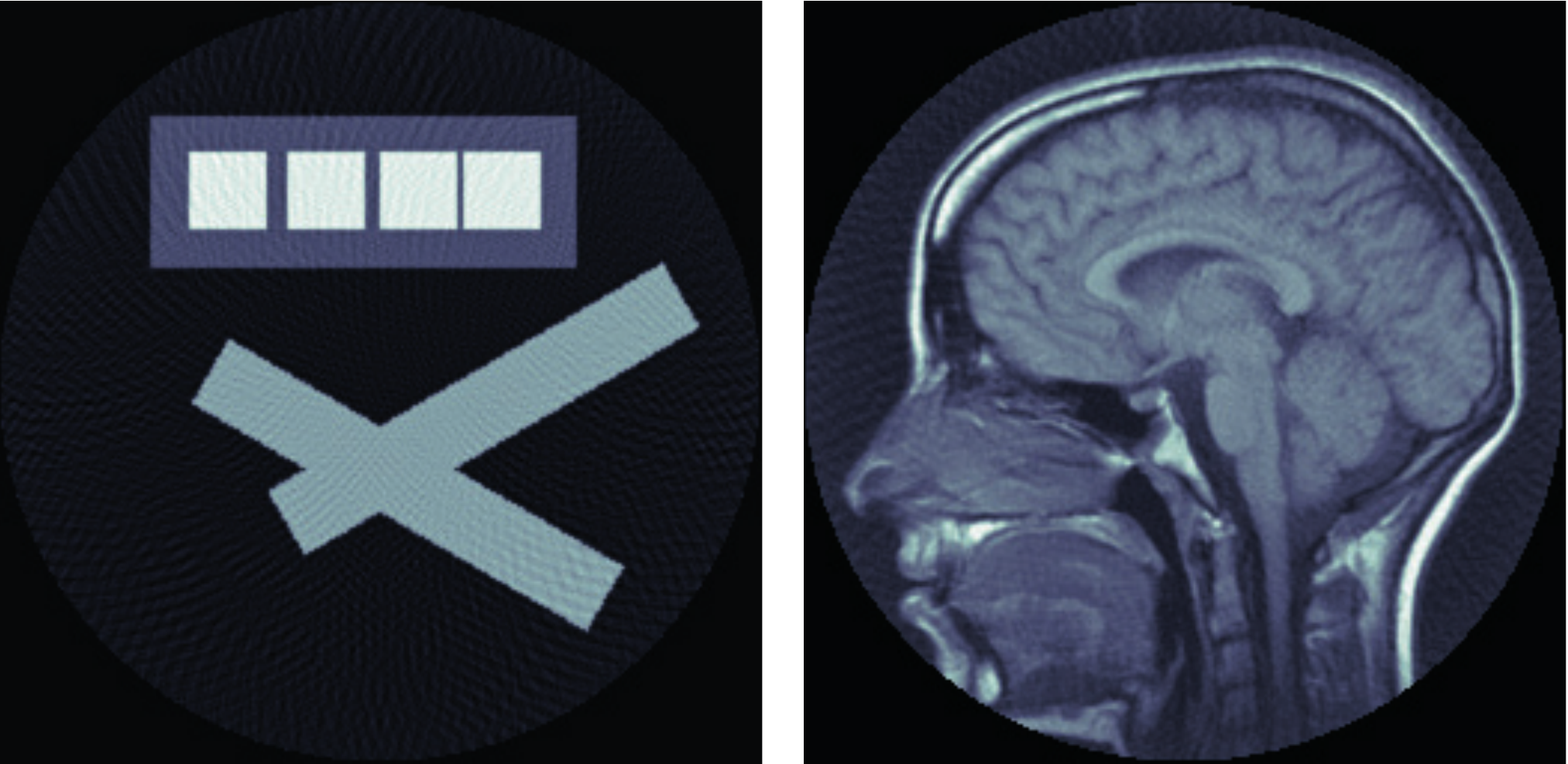}
			\caption{\textbf{Reconstructions from full data (cross phantom and MRI phantom):} \label{fig:fbpRecon} Reconstruction from image from 200
				equispaced (fully sampled) measurements of the  cross and the MRI image
				used  in the numerical experiments.}
		\end{center}
	\end{figure}
\end{psfrags}

CS measurements  $\fullop  \source$ (see \eqref{eq:cs}) have been generated in
two random ways and one  deterministic way. The random matrices $\Am$ have be taken either
as  random Bernoulli matrix with entries $\pm 1 / {\sqrt{m}}$ with equal probability or a  Gaussian random matrix consisting of i.i.d. $\mathcal{N}(0,1/{m})$-Gaussian random variables in each entry. The deterministic subsampling was performed by choosing $m$ equispaced detectors.
In the joint reconstruction with \eqref{eq:prox1}, \eqref{eq:prox2}, the step-size and regularization parameters  are chosen to be $\mu_k = 0.1$, $\alpha =  0.1$ and $\beta = 0.005$; $5000$ iterations  are performed.
For the random  subsampling matrices the recovery guarantees from the theory of CS can be employed, but they do not provably hold for the deterministic subsampling - although the results are equally convincing even for subsampling on the order of \SI{10}{\percent} of the full data, cf. Figure~\ref{fig:reconNoisefree20}.  All results are compared to the standard filtered backprojection (FBP)
reconstruction  applied  to $\Am^T ( \fullop f)$.

\begin{psfrags}
	\begin{figure}[htb!]
		\begin{center}
			\includegraphics[width=1\columnwidth]{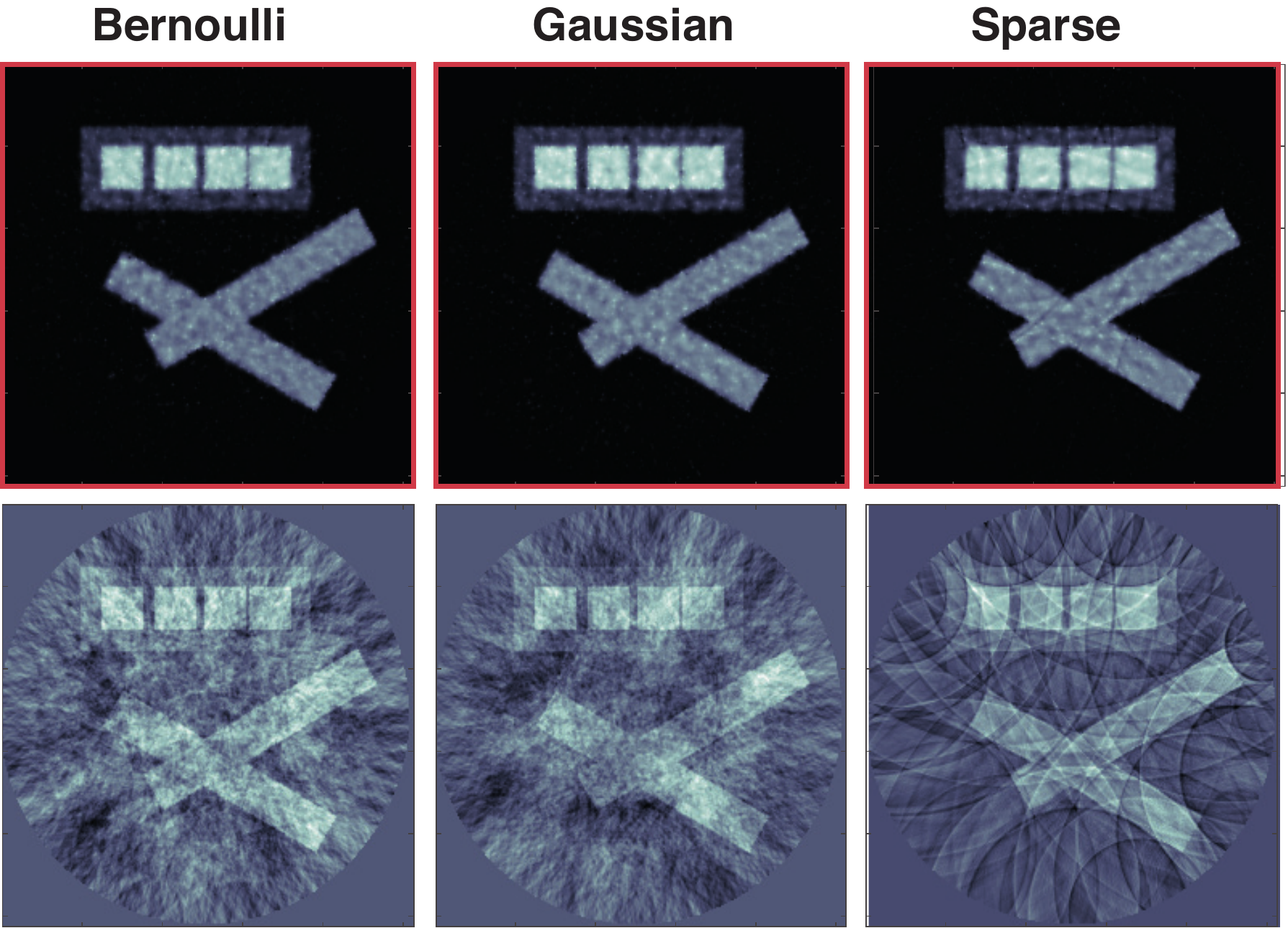}
			\caption{\textbf{Reconstructions from 20 noisefree  measurements (cross phantom):} \label{fig:reconNoisefree20} Reconstruction from $m = 20$
				noisefree CS  measurements (i.e. \SI{10}{\percent} of the fully sampled  data)  using the method
				presented in this article (top) and  FBP (bottom).}
		\end{center}
	\end{figure}
\end{psfrags}

If one increases the number of measurements to about \SI{25}{\percent}  of the data, the reconstruction results become almost indistinguishable from the results obtained from FBP on the full data, cf. Figure~\ref{fig:reconNoisefree50}.

\begin{psfrags}
	\begin{figure}[htb!]
		\begin{center}
			\includegraphics[width=1\columnwidth]{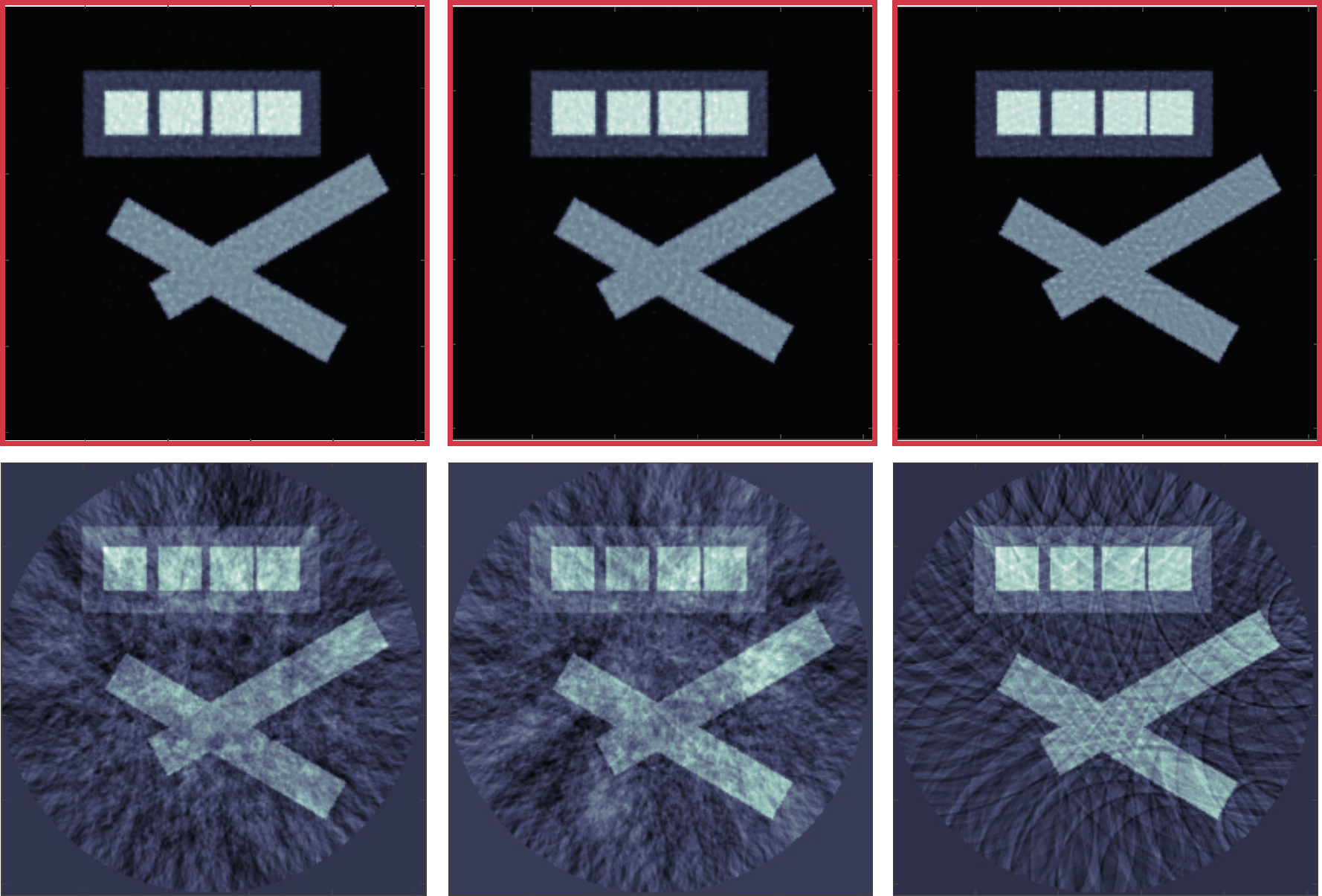}
			\caption{\textbf{Reconstructions from 50 noise-free measurements (cross phantom):}\label{fig:reconNoisefree50} Reconstruction from $m=50$ (i.e. \SI{25}{\percent} of the fully sampled  data) noise-free  CS  measurements using the method presented in this article (top) and  FBP (bottom).}
		\end{center}
	\end{figure}
\end{psfrags}

In Figure~\ref{fig:reconMRI60} the application of the method developed in this article to an MRI
image is presented. As the sparsity of the Laplacian is not as pronounced as in the synthetic example, the MRI image requires more measurements to achieve high qualitative outcome.

\begin{psfrags}
	\begin{figure}[htb!]
		\begin{center}
			\includegraphics[width=0.8\columnwidth]{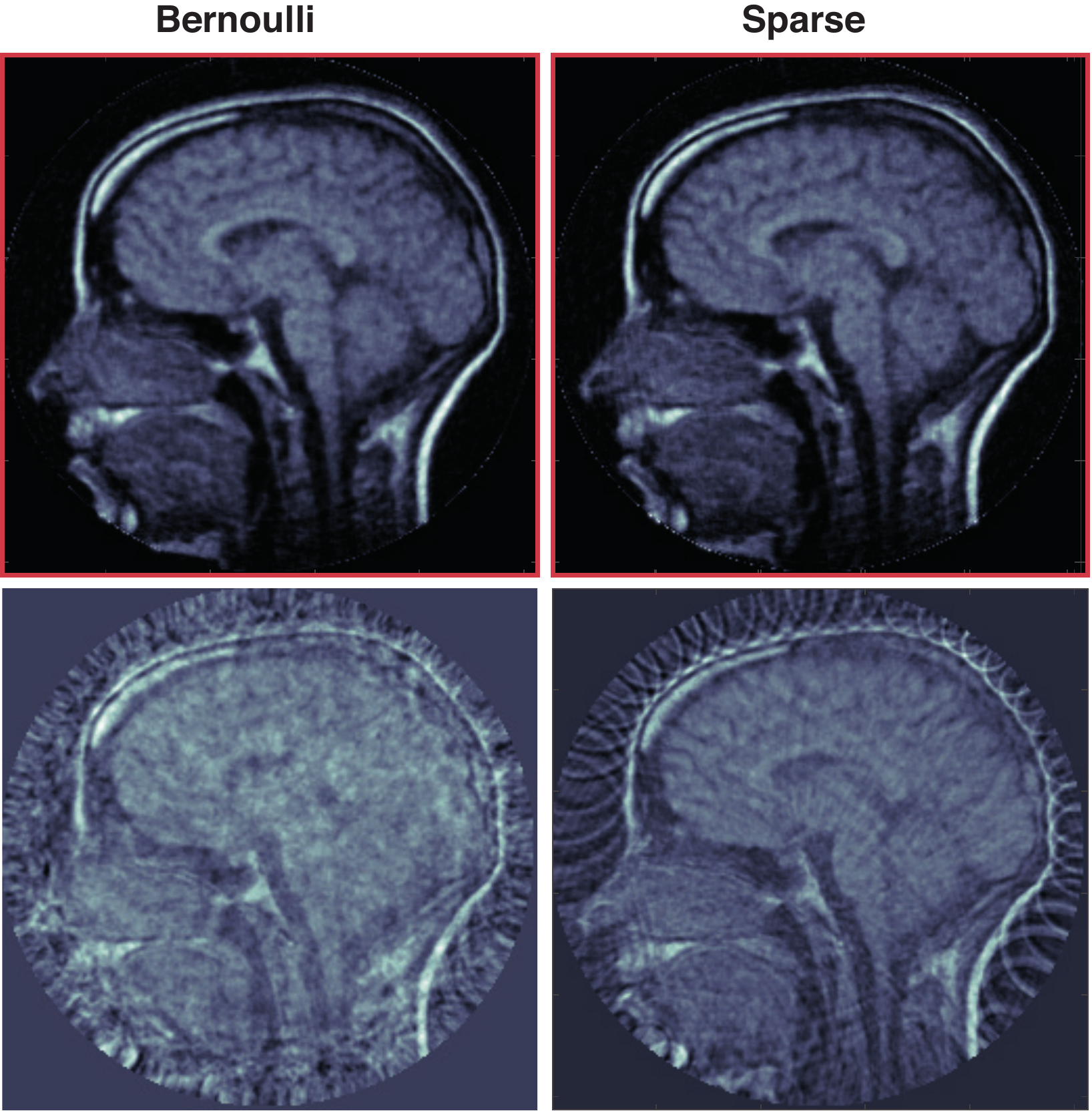}
			\caption{\textbf{Reconstructions from 60 noisefree measurements (MRI phantom):}\label{fig:reconMRI60} Reconstruction of an MRI image from $m=60$ (i.e. \SI{33}{\percent} of the fully sampled  data) synthetically generated noisefree measurements using the method presented in this article (top) and  FBP (bottom).}
		\end{center}
	\end{figure}
\end{psfrags}

For noisy data, the algorithm still produces good results, although more samples need to be taken
to achieve good results. For the synthetic phantom, Gaussian noise amounting to
an SNR of approximately \SI{15}{\percent} was added. The reconstruction results using $m=20$ and $m=50$  measurements  are  depicted in Figure~\ref{fig:reconNoisy20} and Figure~\ref{fig:reconNoisy50}, respectively.

\begin{psfrags}
	\begin{figure}[htb!]
		\begin{center}
			\includegraphics[width=1\columnwidth]{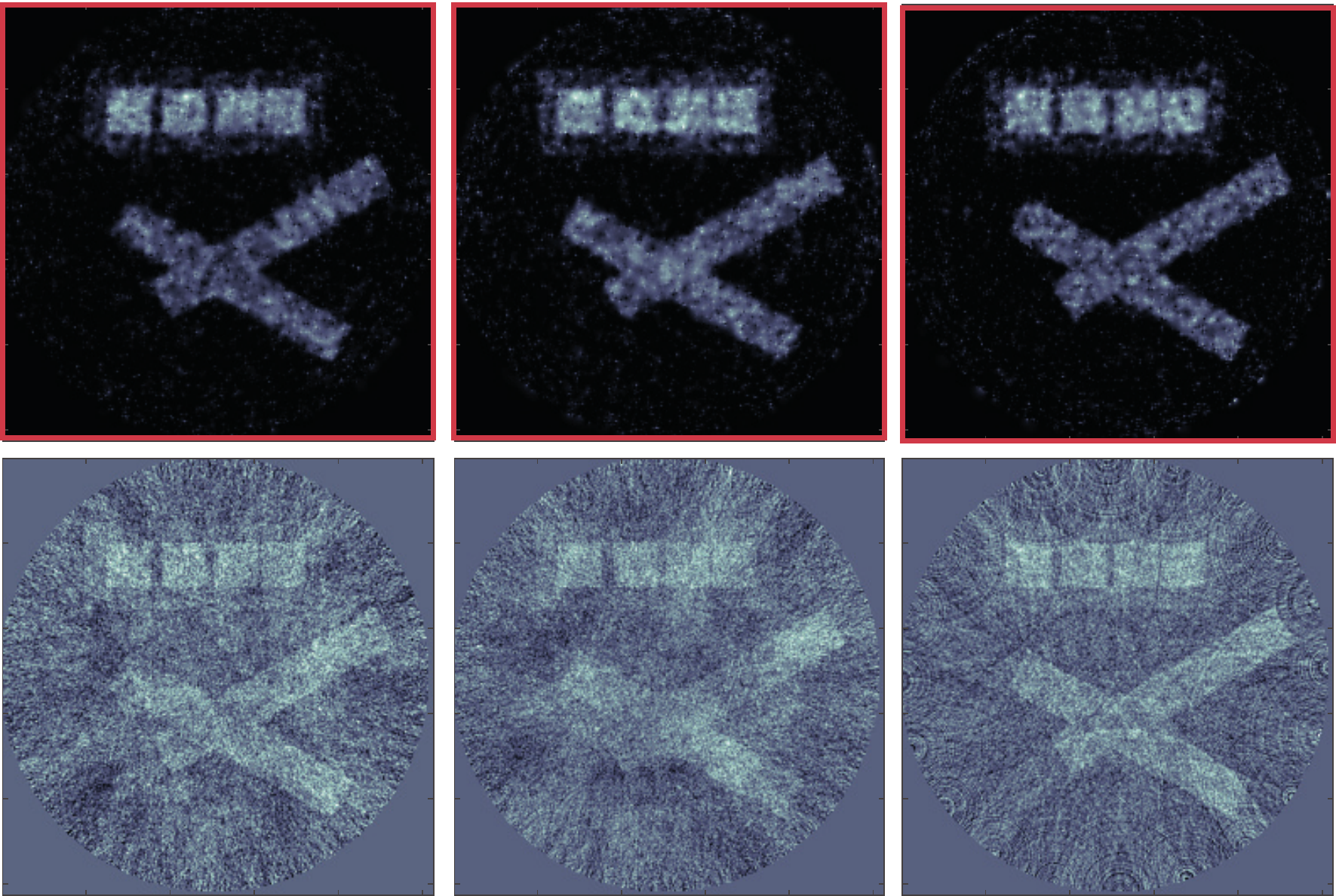}
			\caption{\textbf{Reconstructions from 20 noisy measurements (cross phantom):} \label{fig:reconNoisy20} Reconstruction from $m=20$
				(i.e. \SI{10}{\percent} of the fully sampled data) noisy measurements using the method presented in this article (top) and FBP (bottom).}
		\end{center}
	\end{figure}
\end{psfrags}

\begin{psfrags}
	\begin{figure}[htb!]
		\begin{center}
			\includegraphics[width=1\columnwidth]{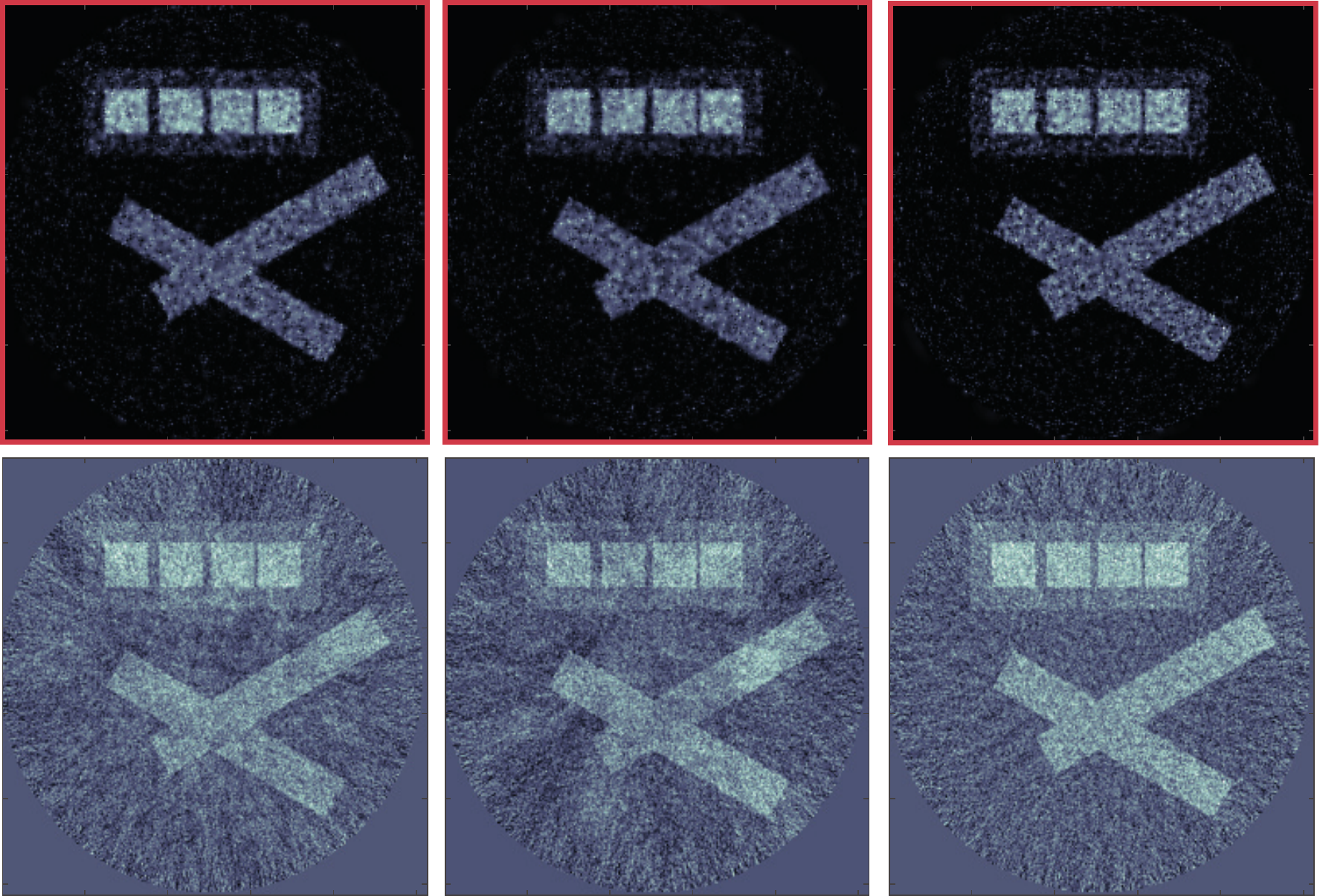}
			\caption{\label{fig:reconNoisy50} \textbf{Reconstructions from 50 noisy   measurements (cross phantom):} Reconstruction from $m=50$ (i.e. \SI{25}{\percent}
				of the fully sampled data) noisy measurements using the method presented in this article (top) and FBP (bottom).}
		\end{center}
	\end{figure}
\end{psfrags}

\subsection{Experimental results}

Experimental data have been acquired by an all-optical photoacoustic projection imaging (O-PAPI) system as described in \cite{Bauer-Marschallinger:17}. The system featured 64 integrating line detector (ILD) elements distributed along a circular arc of radius \SI{4}{cm}, covering an angle of \SI{289}{degree}. For an imaging depth of \SI{20}{mm}, the imaging resolution of the O-PAPI system was estimated to be between
\SI{100}{\mu m} and \SI{260}{\mu m}; see~\cite{Bauer-Marschallinger:17}. PA signals were excited by illuminating the sample from two sides with pulses from a frequency-doubled Nd:YAG laser (Continuum Surelite, \SI{20}{Hz} repetition rate, \SI{6}{ns} pulse duration, \SI{532}{nm} center wavelength) at a fluence of \SI{21}{mJ/cm^2} and recorded by the ILD elements with a sample rate of
\SI{60}{MS/s}.
The sample consisted an approximately triangular shaped piece of ink-stained leaf skeleton, embedded in a cylinder consisting of agarose gel with a diameter of \SI{36}{mm} and a height of \SI{40}{mm}. Intralipid was added to the agarose to increase optical scattering. The strongest branches of the leaf had diameters of approximately \SIrange{160}{190}{\mu m} and the smallest branches of about \SI{50}{\mu m}. Results are only for 2D (projection imaging).

\begin{psfrags}
	\begin{figure}[htb!]
		\begin{center}
			\includegraphics[width=\columnwidth]{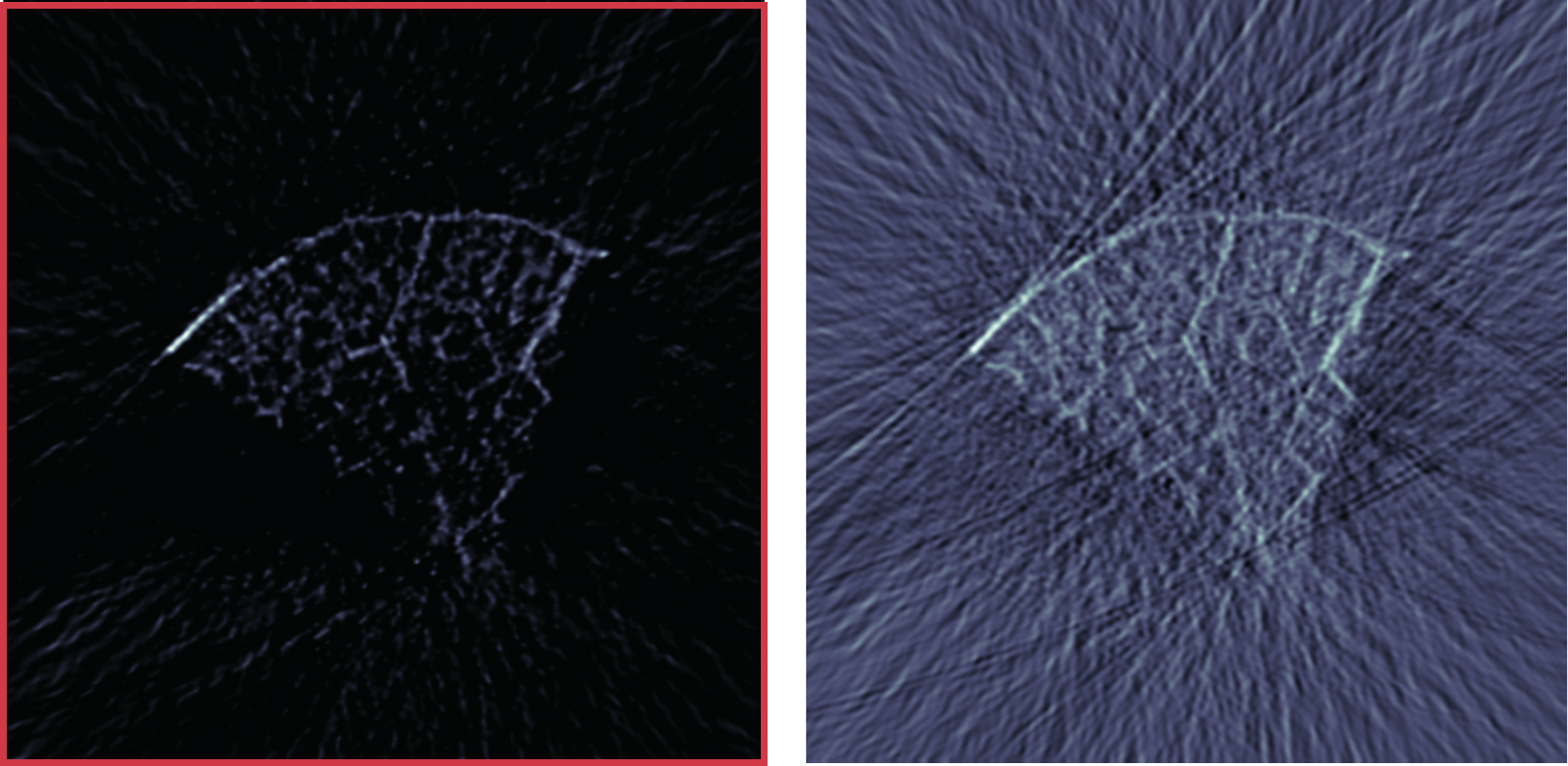}
			\caption{\textbf{Reconstruction from experimental data using
					60 sparse samples:}\label{fig:leaf60}
				Left: PAT image reconstructed with the proposed method.
				Right: FBP reconstruction.}
		\end{center}
	\end{figure}
\end{psfrags}

\begin{psfrags}
	\begin{figure}[htb!]
		\begin{center}
			\includegraphics[width=\columnwidth]{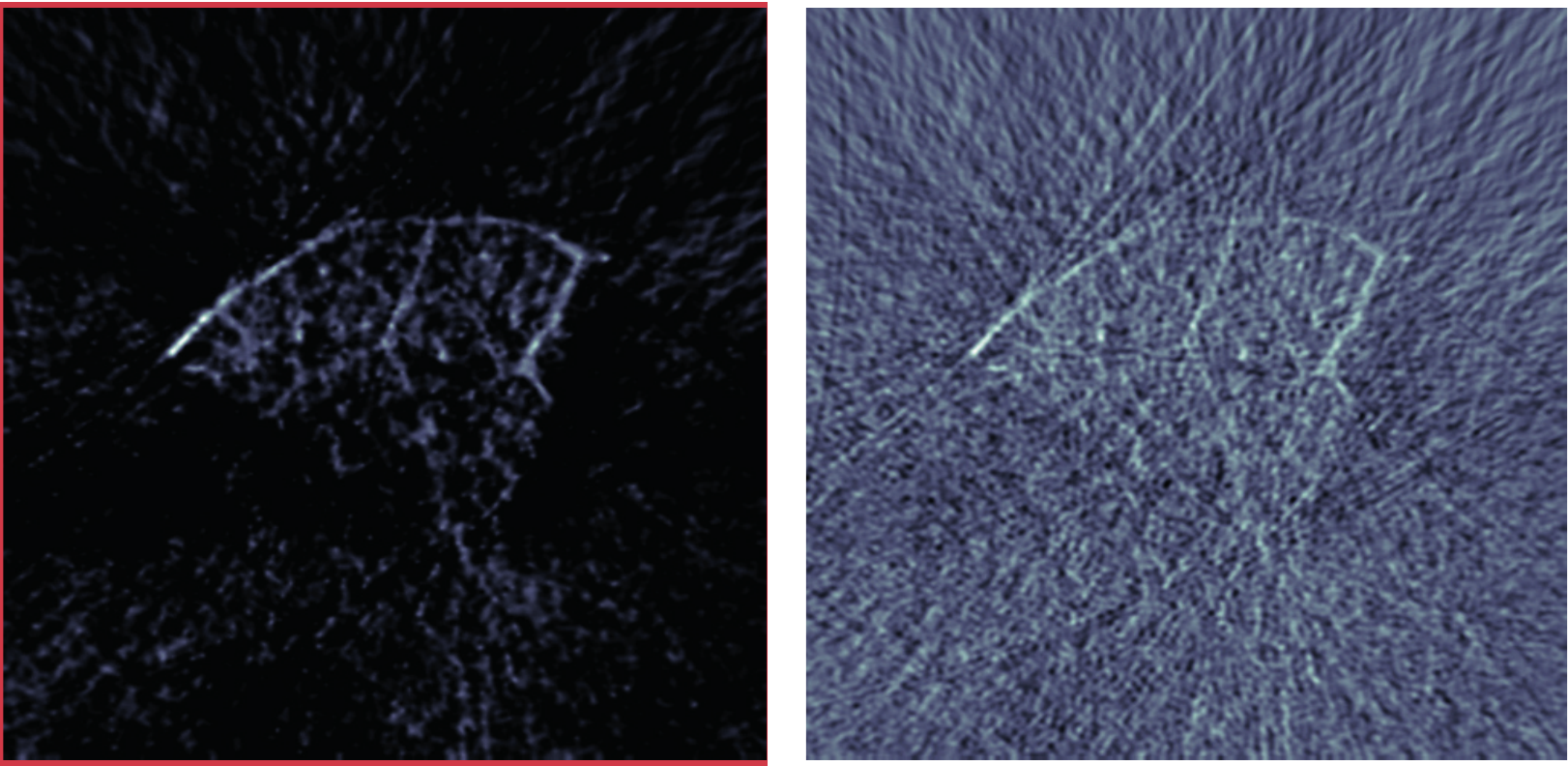}
			\caption{\textbf{Reconstruction using  30 Bernoulli measurements:}\label{fig:leaf30}
				Left: PAT image reconstructed with the proposed method.
				Right: FBP reconstruction.}
		\end{center}
	\end{figure}
\end{psfrags}

Reconstruction results for the leaf phantom from 60 sparsely sampled sensor locations after 500 iterations with the proposed joint minimization algorithm are shown in
Figure~\ref{fig:leaf60}.
For this, the regularization and step-size parameters were chosen as in the previous section. From the experimental data we also generated $m=30$
random Bernoulli measurements. The reconstruction results using this data are
shown in Figure~\ref{fig:leaf30}. For all results, the  PA  source
is displayed on  a \SI{1.6}{cm} $\times$  \SI{1.33}{cm} rectangle with step size
\SI{26}{\mu m} inside  the detection arc.

\section{Conclusion}
\label{sec:conclusion}

In order to achieve high spatial resolution in PAT, standard measurement
and  reconstruction  schemes require a large number of
spatial measurements with high bandwidth detectors. In order to speed up the
measurement process,  systems allowing a large number of parallel measurements are desirable. However  such systems are technically demanding and costly to fabricate.
For example, in PAT with  integrating detectors, the required analog to digital converters
are among the most   costly building blocks. In order to increase measurement speed and to
minimize   system costs, CS aims to reduce the  number of measurements
while preserving  high resolution of the reconstructed image.

One main ingredient enabling CS in PAT is  sparsity of the image to be reconstructed.
To bring sparsity into play, in this paper we introduced a new  approach based on the
commutation  relation $\partial_t^2 \wave  [\source] = \wave  [c^2 \Delta_{\rr} \source]$
between the PAT forward operator $\wave$ and the the Laplacian. We developed a
new reconstruction strategy for jointly reconstructing the pair $[\source,  \Delta_{\rr} \source]$
by minimizing  \eqref{eq:joint-pen} and thereby using sparsity of  $\Delta_{\rr} \source$.
The commutation relation further allows to rigorously study  generalized
Tikhonov regularization of the form $\frac{1}{2} \norm { \fullop\source - \data } _2^2  +
\beta \norm{c\Delta_{\rr} \source}_1  +  I_{C} (\source)$ for CS PAT. Such an analysis
as well as the  development of more efficient numerical minimization schemes are
subjects of  further research.

\section*{Acknowledgments}
	Linh Nguyen's research is supported by the NSF grants DMS 1212125 and DMS 1616904. Markus  Haltmeier and Thomas Berer acknowledge support of the Austrian Science Fund (FWF), project P 30747.
Michael Sandbichler was supported by the Austrian Science Fund (FWF) under Grant no. Y760. Peter Burgholzer and Johannes Bauer-Marschallinger were supported by the strategic economic- and research program ``Innovative Upper Austria 2020'' of the province of Upper Austria. In addition, 
 the computational results presented have been achieved (in part) using the HPC infrastructure LEO of the University of Innsbruck.


\end{document}